\documentclass[11pt]{amsart}

\usepackage{graphicx, graphics}

\usepackage{hyperref}

\emergencystretch=\maxdimen
\usepackage{amsmath}
\usepackage{amsthm}
\usepackage{amsfonts}
\usepackage{color}
\usepackage[english]{babel}
\usepackage[margin=1.5in]{geometry}

\usepackage{enumerate}
\usepackage{subcaption}

        \definecolor{brown}{rgb}{1,0,1}

\usepackage[latin9]{inputenc}
\usepackage{amsmath}
\usepackage{amsfonts}
\usepackage{amssymb}
\usepackage{amsthm}

\numberwithin{equation}{section}

\newtheorem{theo}{Theorem}[section]

\theoremstyle{definition}

\theoremstyle{remark}
\newtheorem{rema}[theo]{Remark}

\newcommand{\nwc}{\newcommand}
\nwc{\eps}{\epsilon}
\nwc{\vareps}{\varepsilon}
\nwc{\Oph}{\operatorname{Op}_\hbar}
\nwc{\ra}{\rangle}
\nwc{\la}{\lambda}

\nwc{\mf}{\mathbf} 
\nwc{\blds}{\boldsymbol} 
\nwc{\ml}{\mathcal} 

\nwc{\defeq}{\stackrel{\rm{def}}{=}}

\nwc{\cE}{\ml{E}}
\nwc{\cN}{\ml{N}}
\nwc{\cO}{\ml{O}}
\nwc{\cP}{\ml{P}}
\nwc{\cU}{\ml{U}}
\nwc{\cV}{\ml{V}}
\nwc{\cW}{\ml{W}}
\nwc{\tU}{\widetilde{U}}
\nwc{\IN}{\mathbb{N}}
\nwc{\IR}{\mathbb{R}}
\nwc{\IZ}{\mathbb{Z}}
\nwc{\IC}{\mathbb{C}}
\nwc{\IT}{\mathbb{T}}
\nwc{\tP}{\widetilde{P}}
\nwc{\tPi}{\widetilde{\Pi}}
\nwc{\tV}{\widetilde{V}}
\nwc{\supp}{\operatorname{supp}}
\nwc{\rest}{\restriction}

\newcommand{\R}{{\mathbb R}}

\newcommand{\Z}{{\mathbb Z}}

\renewcommand{\phi}{\varphi}

\newtheorem{lem}[theo]{{\sc Lemma}}

\title [Applications of quantum ergodicity in nodal sets] {Applications of small scale quantum ergodicity in nodal sets}

\author{Hamid Hezari }
\address{Department of Mathematics, UC Irvine, Irvine, CA 92617, USA} \email{hezari@math.uci.edu}

\begin{document}


\maketitle

\begin{abstract}
The goal of this article is to draw new applications of small scale quantum ergodicity in nodal sets of eigenfunctions. We show that if quantum ergodicity holds on balls of shrinking radius $r(\lambda) \to 0$ then one can achieve improvements on the recent upper bounds of Logunov \cite{La} and Logunov-Malinnikova \cite{LM} on the size of nodal sets, according to a certain power of $r(\lambda)$.  We also show that the order of vanishing results of Donnelly-Fefferman \cite{DF1, DF2}  and Dong \cite{Dong} can be improved. Since by \cite{Han, HeRi} small scale QE holds on negatively curved manifolds at logarithmically shrinking rates, we get logarithmic improvements on such manifolds for the above measurements of eigenfunctions. We also get $o(1)$ improvements for manifolds with ergodic geodesic flows.  Our results work for a full density subsequence of any given orthonormal basis of eigenfunctions.

\end{abstract}

\section{Introduction} Let $(X,g)$ be a smooth compact connected boundaryless Riemannian manifold of dimension $n$.  Suppose $\Delta_g$ is the positive Laplace-Beltrami operator on $(X, g)$ and $\psi_\lambda$ is a sequence of $L^2$ normlazied of eigenfunctions of $\Delta_g$ with eigenvalues $\lambda$.  It was shown in \cite{HeRi} that if for some shrinking radius $r=r(\lambda) \to 0$ and for all geodesic balls $B_r(x)$ one has $ K_1 r^n  \leq || \psi_\lambda ||^2_{B_r(x)} \leq K_2 r^n $, then one gets improved upper bounds \footnote{It is shown by Sogge \cite{So15} that $|| \psi_\lambda ||^2_{B_r(x)} \leq K_2 r^n$ suffices. } of the form $(r^2\lambda)^{ \delta(p)}$ on the $L^p$ norms of $\psi_\lambda$, where ${\delta(p)}$ is Sogge's exponent. The purpose of this article is to prove more applications of small-scale $L^2$ equidistribution of eigenfunctions. We will show that upper bounds on the size of nodal sets as well as the order of vanishing of eigenfunctions can be improved by certain powers of $r$. Since by \cite{HeRi}\footnote{In \cite{Han}, this is proved for $\kappa \in (0, \frac{1}{3n})$} such equidistribution properties hold on negatively curved manifolds \footnote{For a full density subsequence of any given ONB of eigenfunctions.} with $r= (\log \lambda)^{-\kappa}$ for any $\kappa \in (0, \frac{1}{2n})$, we obtain improvements of the results of Logunov \cite{La}, Logunov-Malinnikova \cite{LM}, Donnelly-Fefferman \cite{DF1, DF2}, and Dong \cite{Dong}. We also get slight improvements for quantum ergodic eigenfunctions because roughly speaking they equidistribute on balls of radius $r=o(1)$.

In the following $\mathcal{H}^{n-1} (Z_{\psi_\lambda})$ means the $(n-1)$-dimensional Hausdorff measure of the nodal set of $\psi_\lambda$ denoted by $Z_{\psi_\lambda}$, and $\nu_x(\psi_\lambda)$ means the order of vanishing of $\psi_\lambda$ at a point $x$ in $X$. 

We recall that for $n \geq 3$, a recent result of Logunov \cite{La} gives a polynomial upper bound for $\mathcal{H}^{n-1} (Z_{\psi_\lambda})$ of the form $\lambda^ \alpha$ for some $\alpha > \frac{1}{2}$ depending only on $n$,  and  for $n=2$ another recent result of Logunov-Malinnikova \cite{LM} shows upper bounds of the form $\lambda^{\frac{3}{4}- \beta}$ for some small universal $\beta \in (0, \frac{1}{4})$. Our first result is the following refinement of the results of the above mentioned authors and also the order of vanishing results of Donnelly-Fefferman \cite{DF1, DF2} and Dong \cite{Dong}. 

\begin{theo} \label{UpperBound} Let $(X, g)$ be a boundaryless compact Riemannian manifold with volume measure $dv_g$, and $\psi_\lambda$ be an eigenfunction of $\Delta_g$ of eigenvalue $\lambda >0$. Then there exists $r_0(g) >0$ such that if $ \lambda^{-1/2} < r_0(g)$, and if for some $r \in [\lambda^{-1/2}, r_0(g)]$ and for all geodesic balls $\{B_{r}(x)\}_{x \in X}$ we have
\begin{equation} \label{SmallScaleAssumption}
K_1 r^n\leq \int _{B_{r}(x)} | \psi_{\lambda}|^2 dv_g \leq K_2 r^n,
\end{equation} for some positive  constants $K_1$ and $K_2$ independent of $x$, 
then:

For $n \geq 3$
\begin{equation}\label{RefinedUB} \mathcal{H}^{n-1} (Z_{\psi_\lambda}) \leq c_1 r^{2\alpha-1} \lambda^\alpha, \end{equation}
\begin{equation}\label{RefinedOV} \nu_x(\psi_\lambda) \leq c_2 r \sqrt{\lambda} . \end{equation}
For $n=2$
\begin{equation} \label{RefinedUB2} \mathcal{H}^{1} (Z_{\psi_\lambda}) \leq c_3  r^{\frac{1}{2}-2 \beta}\lambda^{\frac{3}{4}-\beta} , \end{equation}
\begin{equation}\label{RefinedOV2}  \sum_{z \in Z_{\psi_\lambda} \cap B_{r^\frac{1}{2}\lambda^{-\frac{1}{4}}}(x)}  \left ( \nu_z(\psi_\lambda)-1 \right) \leq c_4 r \sqrt{\lambda}. \end{equation}
Here, $\alpha=\alpha(n) >\frac{1}{2}$ and $\beta  \in (0, \frac{1}{4})$ are the universal exponents from \cite{La} and \cite{LM}, and the constants $c_1, c_2, c_3, c_4$ are positive and depend only on $(X, g)$, $K_1$, and $K_2$, and are independent of $\lambda$, $r$, and $x$. Note that the quantity on the left hand side of (\ref{RefinedOV2}) counts the number of singular points $\mathcal S=\{ \psi_\lambda= |\nabla \psi_\lambda |=0 \}$ in geodesic balls of radius $r^\frac{1}{2}\lambda^{-\frac{1}{4}}$.
\end{theo}

Combining this with our result with Rivi\`ere \cite{HeRi}, which states that  on negatively curved manifolds (\ref{SmallScaleAssumption}) holds with $r= (\log \lambda)^{-\kappa}$ for any $\kappa \in (0, \frac{1}{2n})$, the following unconditional results on such manifolds are immediate.

\newpage

\begin{theo}\label{UpperBound-NC}
Let $(X, g)$ be a boundaryless compact connected smooth Riemannian manifold of dimension $n$, with negative sectional curvatures.  Let $\{ \psi_{\lambda_j} \}_{j \in \IN}$ be any ONB of $L^2(X)$ consisting of eigenfunctions of $\Delta_g$ with eigenvalues $\{\lambda_j\}_{j \in \IN}$. Let $\epsilon >0$ be arbitrary. Then there exists $S \subset \IN$ of \textit{full density} \footnote{It means that $ \lim_{N \to \infty} \frac{1}{N} \text{card} \big ( S \cap [1, N] \big )  =1$.} such that for $j \in S$:
$$ \text{if} \; \;  n\geq 3: \qquad \mathcal{H}^{n-1} (Z_{\psi_{\lambda_j}}) \leq  c_1 (\log \lambda_j)^{\frac{1-2\alpha}{2n} +\epsilon} \lambda_j^\alpha,$$
$$ \text{if} \; \;  n=2 : \qquad  \mathcal{H}^{1} (Z_{\psi_{\lambda_j}}) \leq  c_3 (\log \lambda_j)^{-\frac{1}{8}+\frac{\beta}{2} + \epsilon} \lambda_j^{\frac{3}{4}-\beta}.$$ 
In addition, for all dimensions 
$$ \qquad \qquad \quad \nu_x(\psi_{\lambda_j}) \leq  c_2 (\log \lambda_j)^{-\frac{1}{2n}+\epsilon} \sqrt{\lambda_j}\; .$$ \end{theo}

We repeat that here $\alpha=\alpha(n) >\frac{1}{2}$ and $\beta  \in (0, \frac{1}{4})$ are the universal exponents from \cite{La} and \cite{LM}, and $c_1, c_2, c_3$  depend only on $(X, g)$ and $\epsilon$. 

We will also prove the following $o(1)$ improvements for quantum ergodic sequences of eigenfunctions. In fact equidistribution on $X$ (instead of the phase space $S^*X$) suffices.  

\begin{theo} \label{UpperBoundQE} Let $(X, g)$ be a boundaryless compact connected smooth Riemannian manifold of dimension $n$.  Let $\{ \psi_{\lambda_j}\}_{j \in S}$ be a  sequence of eigenfunctions of $\Delta_g$ with eigenvalues $\{\lambda_j\}_{j \in S}$ such that for all $r \in (0, \text{inj}(g)/2)$ and all $x \in X$
\begin{equation} \label{QEonX}
\int_{B_r(x)} |\psi_{\lambda_j}|^2 \to \frac{\text{Vol}_g(B_r(x))}{\text{Vol}_g(X)}, \qquad  \lambda_j \xrightarrow{j \in S} \infty.
\end{equation}
Then, along this sequence, for $n \geq 3$ 
$$ \mathcal{H}^{n-1} (Z_{\psi_{\lambda_j}}) = o ( \lambda_j^\alpha), $$ and for $n=2$
$$  \mathcal{H}^{1} (Z_{\psi_{\lambda_j}})=  o (\lambda_j^{\frac{3}{4}-\beta}). $$ Also in all dimensions 
$$ \qquad \qquad \qquad \quad \quad \nu_x(\psi_{\lambda_j})=o(\sqrt{\lambda_j}) \qquad (\text{uniformly in } \;  x ).$$
\end{theo}
In particular the above theorem holds for manifolds with ergodic geodesic flows by the quantum ergodicity theorem of Shnirelman-Colin de Verdi\`ere-Zelditch \cite{Sh}-\cite{CdV}-\cite{Ze87}. Hence given any ONB of eigenfunctions on such a manifold one can pass to a full density subsequence where (\ref{QEonX}), whence Theorem \ref{UpperBoundQE} holds. 

\begin{rema} We point out that the equidistribution property (\ref{QEonX}), which is weaker than quantum ergodicity, holds for some non-ergodic manifolds such as the flat torus and the rational polygons (see \cite{MaRu}, \cite{Ri}, and \cite{Taylor}). 
\end{rema}

\subsection{Main idea}
The major idea in proving our upper bounds is to lower the doubling index 
$$ N(B_s(x)):= \log \left ( \frac{\sup_{B_{2s}(x)} |\psi_{\lambda}|^2} { \sup_{B_{s}(x)} |\psi_{\lambda}|^2}  \right )$$ under the assumption
$$K_1 r^n\leq \int _{B_{r}(x)} | \psi_{j}|^2 \leq K_2 r^n.$$ 
We recall that Donnelly-Fefferman \cite{DF1} showed that an eigenfunction $\psi_\lambda$ of $\Delta_g$ with eigenvalue $\lambda$ satisfies
$$ N(B_s(x)) \leq c \sqrt{\lambda},$$ for all $s < s_0$ where $s_0$ and $c$ depend only on $(X, g)$.   
We will prove in Lemma \ref{RefinedDoublingEstimate} that
$$N(B_s(x)) \leq c \, r \sqrt{\lambda},$$ for all $s < 10 r$ where $c$  depends only on $(X, g)$. We then apply this modified growth estimate to the proofs of \cite{La, LM, DF1, DF2, Dong} to obtain our improvements.

\subsection{\textbf{Background on the size of nodal sets}}
For any smooth compact connected Riemannian manifold $(X, g)$ of dimension $n$, Yau's conjecture states that there exist constants $c>0$ and $C>0$ independent of $\lambda$ such that
$$c \sqrt{\lambda} \leq \mathcal H ^{n-1} (Z_{\psi_\lambda}) \leq C \sqrt{\lambda}.$$
The conjecture was proved by Donnelly and Fefferman \cite{DF1} in the real analytic case. In dimension two and the $C^\infty$ case,  Br\"uning \cite{B} and Yau proved the lower bound $c \sqrt{\lambda}$. Until the recent result of Logunov-Mallinikova \cite{LM} the best upper bound in dimension two was $C\lambda^{3/4}$, which was proved independently by Donnelly-Fefferman \cite{DF2} and Dong \cite{Dong}. The result of \cite{LM} gives $C\lambda^{3/4-\beta}$ for some small universal constant $\beta <\frac{1}{4}$. In dimensions $n \geq 3$ until very recently, the best lower bound was $c \lambda^{\frac{3-n}{4}}$, proved \footnote{Different proofs were given later by \cite{HW12, HS12, SZ12} based on the earlier work \cite{SZ11},  and by \cite{St} using heat equation techniques. Also logarithmic improvements of the form $\lambda^{\frac{3-n}{4}} (\log \lambda)^ \alpha$ were given in \cite{HeRi} on negatively curved manifolds and in \cite{BlSo} on non-positively curved manifolds.} by Colding-Minicozzi \cite{CM}. However, a recent breakthrough result of Logunov \cite{Lb} proves the lower bound $c \sqrt{\lambda}$ for all $n \geq 3$. Also another result of \cite{La} shows a polynomial upper bound $C \lambda^\alpha$ for some $\alpha >\frac{1}{2}$ which depends only on $n$. The best upper bound before this was the exponential bound $e^{c \sqrt{\lambda} \log \lambda}$ of Hardt-Simon \cite{HaSi}.   

\subsection{\label{SSQE} \textbf{Background on small scale quantum ergodicity}} First, we recall that the quantum ergodicity result of Shnirelman-Colin de Verdi\`ere-Zelditch \cite{Sh, CdV, Ze87} implies in particular that if the geodesic flow of a smooth compact Riemannian manifold without boundary is ergodic then for any ONB $\{ \psi_{\lambda_j} \}_{j=1}^\infty$ consisting of the eigenfunctions of $\Delta_g$, there exists a full density subset $S \subset \IN$ such that for any $r < \text{inj}(g)$, independent of $\lambda_j$, one has \begin{equation}\label{QE} || \psi_{\lambda_j}||^2_{L^2(B_{r}(x))} \sim \frac{\text{Vol}_g(B_r(x))}{\text{Vol}_g(X)}, \qquad \text{as}\quad  \lambda_j \to \infty, \quad j \in S.  \end{equation}
The analogous result on manifolds with piecewise smooth boundary and with ergodic billiard flows was proved by \cite{ZZ}.  

The small scale equidistribution  problem asks whether (\ref{QE}) holds for $r$ dependent on $\lambda_j$. A quantitative QE result of Luo-Sarnak \cite{LuSa} shows that the Hecke eigenfunctions on the modular surface satisfy this property along a density one subsequence for $r=\lambda^{- \kappa}$ for some small $\kappa>0$.  Also, under the Generalized Riemann Hypothesis, Young \cite{Yo} has proved that small scale equidistribution holds for Hecke eigenfunctions for $r = \lambda^{-1/4 + \epsilon}$.

This problem was studied in \cite{Han} and \cite{HeRi} for the eigenfunctions of negatively curved manifolds. To be precise, it was proved that on compact negatively curved manifolds without boundary, for any $\epsilon >0$ and any ONB $\{ \psi_{\lambda_j} \}_{j=1}^\infty$ of $L^2(X)$ consisting of the eigenfunctions of $\Delta_g$, there exists a subset $S \subset \IN$ of full density such that for all $x \in X$ and $j \in S$:
\begin{equation} \label{QENC} \quad K_1 r^n \leq  || \psi_{\lambda_j}||^2_{L^2(B_{r}(x))} \leq K_2 r^n, \qquad \text{with} \;\; r=(\log \lambda_j)^{-\frac{1}{2n} +\epsilon}, \end{equation} for some positive constants $K_1, K_2$ which depend only on $(X, g)$ and $\epsilon$. The same result was proved in \cite{Han} for $r=(\log \lambda_j)^{-\frac{1}{3n} +\epsilon}$. 

We also point out that although eigenfunctions on the flat torus $\R^n / \Z^n$  are not quantum ergodic, however they equidistribute on the configuration space $\R^n / \Z^n$ (see \cite{MaRu}, and also \cite{Ri} and \cite{Taylor} for later proofs). So one can investigate the small scale equidistribution property for toral eigenfunctions. It was proved in \cite{HeRiTorus} that a commensurability of $L^2$ masses such as (\ref{QENC}) is valid for a full density subsequence with $r = \lambda^{-1/(7n+4)}$. Lester-Rudnick \cite{LeRu} improved this rate of shrinking to  $r= \lambda^{- \frac{1}{2n-2} +\epsilon}$, and in fact they proved that the stronger statement (\ref{QE}) holds.  They also showed that their results are almost \footnote{They show that the equidistribution property fails for $r= \lambda^{- \frac{1}{2n-2} -\epsilon}$} sharp. The case of interest is $n=2$, which gives $r= \lambda^{-1/2 +\epsilon}$. A natural conjecture is that this should be the optimal rate of shrinking on negatively curved manifolds. A recent result of \cite{Han16} proves that random eigenbases on the torus enjoy small scale QE for $r= \lambda^{-\frac{n-2}{4n} +\epsilon}$, which is better than \cite{LeRu} for $n \geq 5$.

\subsection{\textbf{Some remarks}}

\begin{rema} In our proof we have used both \textit{local and global harmonic analysis} (see \cite{Ze08} for background). The local analysis is used in the works of \cite{La, LM}, and the global analysis is used in \cite{HeRi} to obtain equidistribution on small balls. We emphasize that our improvements of \cite{La, LM} are robust, in the sense that any upper bounds of the form $\lambda^{\alpha}$ for $\alpha > \frac{1}{2}$ that are resulted from a purely local analysis of eigenfunctions can be improved using our combined method. 
\end{rema}

\begin{rema}
The most important assumption of Theorem \ref{UpperBound} is the lower bound $K_1 r^n\leq \int _{B_{r}(x)} | \psi_{\lambda}|^2$ and the upper bound in (\ref{SmallScaleAssumption}) can be discarded at the expense of messy estimates in Theorem  \ref{UpperBound}.  In fact using Sogge's ``trivial local $L^2$ estimates" \cite{So15}, which asserts that one always has $ \int _{B_{r}(x)}|\psi_{\lambda}|^2 \leq K_2 r$, we can still prove modified doubling estimates of the form 
$$\sup _{B_{2s}(x)}|\psi_{\lambda}|^2  \leq r^{-b} e^{c r \sqrt{ \lambda} }\sup_{B_{s}(x)}|\psi_{\lambda}|^2, \quad  \text{for some  $b =b(n) >0$ and all $ s <10r$} .$$  We can use this inequality and obtain estimates similar to those in Theorem \ref{UpperBound}, however we have not done so for the sake of more polished estimates. Another reason that we have not discarded the assumption $\int _{B_{r}(x)}|\psi_{\lambda}|^2 \leq K_2 r^n$  is that all the examples (such as QE eigenfunctions) for which we know the lower bounds are satisfied, also satisfy the upper bounds in (\ref{SmallScaleAssumption}). 
\end{rema}

\begin{rema} As we discussed in the previous section a result of \cite{LuSa} implies that small scale QE holds for a full density subsequence of Hecke eigenfunctions on the modular surface, for balls of radius $r= \lambda^{-\kappa}$ for some explicitly calculable $\kappa >0$. Hence using (\ref{RefinedOV}), we get upper bounds of the form $\lambda^{\frac{1}{2} -\kappa}$ on the order of vanishing of these eigenfunctions. We could not find any arithmetic results in the literature discussing improvements on the upper bound $\sqrt{\lambda}$ of Donnelly-Fefferman.   Of course a natural conjecture to impose is that for Hecke eigenfunctions $\nu_x( \psi_\lambda) \leq c \lambda^\epsilon.$ Although the available graphs of nodal lines of Hecke eigenfunctions with high energy do not show any singular points i.e. places where nodal lines intersect each other, but there are many almost intersecting nodal lines.

\end{rema}

\begin{rema} By our discussion in the previous section on the work of \cite{LeRu}, and using (\ref{RefinedOV}), we get that for a full density subsequence of toral eigenfunctions on the $2$-torus, we have $\nu_x(\psi_\la) \leq c \lambda^\epsilon$. However, it is proved in \cite{BoRu} that $\nu_x(\psi_\la) \leq c \lambda^{\frac{1}{ \log \log \lambda}}$ for all eigenfunctions on $\mathbb T^2$. 

\end{rema}

\begin{rema} Theorem \ref{UpperBound} is local in nature, meaning that if the eigenfunctions satisfy (\ref{SmallScaleAssumption}) for balls centered on an open set, then we get the upper bounds in this theorem on that open set. In particular we get all the upper bounds in  Theorem \ref{UpperBoundQE} for eigenfunctions on ergodic billiards (and also rational polygons) as long as we stay a positive distance away from the boundary. One would expect that the results of \cite{La} and \cite{LM} can be extended to the eigenfunctions of the Laplacian on manifolds with boundary (with Dirichlet or Neumann boundary conditions) using the method of \cite{DF3}. 
\end{rema}


\section{Proofs of upper bounds for nodal sets and order of vanishing}
The following lemma is the main ingredient of the proofs. It gives improved growth estimates for eigenfunctions under our $L^2$ assumption on small balls. 

\begin{lem}\label{RefinedDoublingEstimate}
Let $(X, g)$ be a smooth Riemannian manifold, $p \in X$ a fixed point, and $R>0$ be a fixed radius so that the geodesic ball $B_{2R}(p)$ is embedded. Then there exists $r_0(g)$ such that the following statement holds:

Suppose $\lambda^{-1/2} \leq r_0(g)$ and $\psi_\lambda$ is a smooth function such that $\Delta_g \psi_\lambda = \lambda \psi_\lambda$ on $B_{2R}(p)$. If for some $r \in [\lambda^{-1/2}, r_0(g)]$ and all  $x \in B_{R}(p)$,
\begin{equation}\label{L2assumption}
 \quad  K_1 r^n\leq \int _{B_{r}(x)} | \psi_{\lambda}|^2 \leq K_2 r^n
\end{equation} holds for some  positive constants $K_1$ and $K_2$ independent of $x$,  then one has the following refined doubling estimates
\begin{equation} \label{L2DoublingEstimate}
 \delta \in (0, 10r), x \in B_{\frac{R}{2}}(p): \quad \int _{B_{2 \delta}(x)} | \psi_{\lambda}|^2 \leq e^{c \, r \sqrt{\lambda}} \int _{B_{ \delta}(x)} | \psi_{\lambda}|^2,
\end{equation}
\begin{equation} \label{LinftyDoublingEstimate}
\delta \in (0, 10r), x \in B_{\frac{R}{2}}(p): \quad \sup _{B_{2 \delta}(x)} | \psi_{\lambda}|^2 \leq e^{c \, r \sqrt{\lambda}} \sup_{B_{ \delta}(x)} | \psi_{\lambda}|^2.
\end{equation} We also have
\begin{equation} \label{L2LowerBound}
\delta \in (0, r/2), x \in B_{\frac{R}{2}}(p): \quad \frac{1}{\delta^n} \int _{B_{\delta}(x)} | \psi_{\lambda}|^2 \geq \left ( {\frac{r}{\delta}} \right )^{-c \, r \sqrt{\lambda}},
\end{equation} 
\begin{equation} \label{LinftyLowerBound}  \delta \in (0, r/2), x \in B_{\frac{R}{2}}(p): \quad \sup _{B_{\delta}(x)} | \psi_{\lambda}|^2 \geq  \left ( {\frac{r}{\delta}} \right )^{-c \, r \sqrt{\lambda}}.
\end{equation}
Here $c$ is positive and is uniform in $x$, $r$, $\delta$, and $\lambda$, but depends on $K_1$, $K_2$, and $(B_{2R}(p), g)$.
\end{lem}

\begin{proof} We will give two proofs for (\ref{LinftyDoublingEstimate}). All other statements will follow from this as we will show.  The first proof of (\ref{LinftyDoublingEstimate}) follows from a rescaling argument applied to the following theorem of Donnelly-Fefferman, which is a purely local result based on Carleman estimates. The second proof relies on a theorem of Mangoubi \cite{Ma}. 

\begin{theo} [Donnelly-Fefferman \cite{DF1}, Proposition 3.10.ii] Let $(\tilde X, \tilde g)$ be a smooth Riemannian manifold, $p \in \tilde X$ a fixed point, and $\tilde R >0$ a fixed radius such that the $\tilde g$-geodesic ball $\tilde B_{2 \tilde R}(p)$ is embedded. Let $\psi_{\tilde \lambda}$ be a smooth function such that for some $\tilde \lambda \geq 1$ we have $\Delta_{\tilde g} \psi_{\tilde \lambda} = \tilde \lambda \psi_{\tilde \lambda}$ on $\tilde B_{2 \tilde R}(p)$. Then there exists a suitably small $h_0(\tilde g) >0$ such that for all $ h \leq h_0(\tilde g), \delta < h/2$, and $x \in \tilde B_{\frac{\tilde R}{2}}(p)$:

\begin{equation}\label{DF-DoublingEstimate}
\sup _{\tilde B_{2 \delta }(x)} | \psi_{ \tilde \lambda}|^2 \leq e^{\kappa_1 \, \sqrt{\tilde \lambda}} \left ( \frac{\sup _{\tilde B_{ h }(x)} | \psi_{ \tilde \lambda}|^2 }{  \sup _{\tilde B_{h/5 }(x) \backslash \tilde B_{h/10}(x)} | \psi_{ \tilde \lambda}|^2 }    \right )^{\kappa_2} \sup_{ \tilde B_{\delta}(x)} | \psi_{ \tilde \lambda}|^2 .
\end{equation}
The constant $h_0(\tilde g)$ is controlled by $ \tilde R$ and  the reciprocal of the square root of $\sup_{\tilde B_{2 \tilde R}(p)} |\text{Sec}(\tilde g )|$, and the constants $\kappa_1$ and $\kappa_2$ are controlled by  $\sup_{ \tilde B_{2 \tilde R}(p)} |\text {Sec}(\tilde g )|$.
\end{theo}
To prove our lemma, we define $(\tilde X, \tilde g)= ( X, \frac{1}{r^2}g)$, and $\tilde R = \frac{1}{r} R$. Then the equation  $$\Delta_g \psi_\lambda = \lambda \psi_\lambda \qquad  \text{on} \quad B_{2R}(p),$$ becomes 
$$\Delta_{\tilde g} \psi_{\tilde \lambda} = \tilde \lambda \psi_{\tilde \lambda} \qquad \text{on} \quad \tilde B_{2 \tilde R}(p),$$ with $$\tilde \lambda = r^2\lambda  \qquad \text{and} \qquad \psi_{\tilde \lambda} =\psi_\lambda.$$  We then note that by \cite{DF1}, although not explicitly stated, we have
$$h_0(\tilde g) = {C} \min \Big ( \tilde R/2, (\sup_{ \tilde B_{2 R}(p)} |\text {Sec}( \tilde g )|)^{-1/2} \Big) =\frac{C}{r} \min \Big (R/2, (\sup_{ B_{2 R}(p)} |\text {Sec}(g )|)^{-1/2} \Big ), $$ for some suitably small $C$ that is uniform in $r$. Hence if we set $$r_0(g) \leq  \frac{C}{20}\min \Big( R/2, (\sup_{ B_{2 R}(p)} |\text {Sec}(g )|)^{-1/2} \Big )$$ then for all $r \leq r_0(g)$ we have $h_0(\tilde g) \geq 20$, and therefore we can choose $h=20$.  As a result, by Theorem \ref{DF-DoublingEstimate}
$$
\small{\delta \in (0, 10), x \in \tilde B_{\frac{\tilde R}{2}}(p):} \quad \sup _{ \tilde B_{2 \delta }(x)} | \psi_{ \tilde \lambda}|^2 \leq e^{\kappa_1 \, \sqrt{\tilde \lambda}} \left ( \frac{\sup _{\tilde B_{ 20 }(x)} | \psi_{ \tilde \lambda}|^2 }{  \sup _{\tilde B_{4 }(x) \backslash \tilde B_{2}(x)} | \psi_{ \tilde \lambda}|^2 }    \right )^{\kappa_2} \sup_{\tilde B_{\delta}(x)} | \psi_{ \tilde \lambda}|^2 .
$$
Writing this inequality with respect to the metric $g$ we get 
\begin{equation} \label{Doubling}
\small {\delta \in (0, 10r), x \in B_{\frac{R}{2}}(p):}  \quad \sup _{B_{2 \delta }(x)} | \psi_{\lambda}|^2 \leq e^{\kappa_1 \, r \sqrt{\lambda}} \left ( \frac{\sup _{B_{ 20 r }(x)} | \psi_{ \lambda}|^2 }{  \sup _{B_{4 r}(x) \backslash B_{2r}(x)} | \psi_{\lambda}|^2 }    \right )^{\kappa_2} \sup_{B_{\delta}(x)} | \psi_{\lambda}|^2 .
\end{equation} 

\begin{rema}We emphasize that since $|\text{Sec}( \tilde g)| = r^2 |\text{Sec}(g)|$, and since $r$ is bounded by $r_0(g)$, the constants $\kappa_1$ and $\kappa_2$ can be chosen independently from $r$.  \end{rema}

We now bound the expression in parenthesis using our local $L^2$ assumptions (\ref{L2assumption}).  First we find $y$ such that $$B_r(y) \subset B_{4 r}(x) \backslash B_{2r}(x).$$ Since by assumption $\int _{B_{r}(y)} | \psi_{j}|^2 \geq K_1 r^n$, we must have $$\sup _{B_{4 r}(x) \backslash B_{2r}(x)} | \psi_{ \lambda}|^2 \geq \sup_{B_r(y)} | \psi_{ \lambda}|^2 \geq \frac{r^n}{\text{Vol}(B_r(y))} K_1. $$  By making $r_0(g)$  sufficiently smaller, we obtain that for any $r \leq  r_0(g)$ which satisfies (\ref{L2assumption}),  we have
\begin{equation}\label{Denominator} \sup _{B_{4 r}(x) \backslash B_{2r}(x)} | \psi_{ \lambda}|^2  \geq a K_1,\end{equation}
for some constant $a$ which is uniform in $x \in B_{R/2}(p)$, $r \in (0, r_0(g))$, and $\lambda$.
For the numerator in the parenthesis we claim that \footnote{In fact when $(X, g)$ is a \textit{closed manifold} the better estimate $ b K_2 (r \sqrt{\lambda})^{n-1}$ holds using Sogge's local $L^\infty$ estimates \cite{So15}, but we do not need this better estimate.}
\begin{equation} \label{LocalLinfty} \sup _{B_{ 20 r }(x)} | \psi_{ \lambda}|^2 \leq b K_2 (r \sqrt{\lambda})^n, \end{equation} for some constant $b$ which is uniform in $x \in B_{R/2}(p)$, $r \in (0, r_0(g))$ and $\lambda$. To prove (\ref{LocalLinfty}) we cover $B_{20r}(x)$ using balls of radius $\frac{r}{2}$. It is therefore enough to show that 
\begin{equation} \label{GT} \sup _{B_{r/2}(y)} | \psi_{ \lambda}|^2 \leq b \lambda^{n/2} \sup _{z \in B_{r/2}(y)} \int_{B_r(z)} | \psi_{ \lambda}|^2, \end{equation} for some $b$ that is uniform in $y$, $r$, and $\lambda$.   This estimate however follows from standard elliptic estimates (see for example \cite{GT}, Theorem 8.17 and Corollary 9.21) which asserts that there exists $a_0 <1$ suitably small such that for  $z \in B_R(p)$ we have
\begin{equation} \label{EllipticEstimate} \forall s \in(0,  a_0 \lambda^{-\frac{1}{2}}]: \quad \sup _{B_{ s/2}(z)} | \psi_{ \lambda}|^2 \leq b_0 s^{-n} \int_{B_{s}(z)} | \psi_{ \lambda}|^2, \end{equation} for some $b_0$ which is uniform in $\lambda$, $z$, and $s$. Since $ \lambda^{-1/2} \leq r$, we have $B_{a_0 \lambda^{-\frac{1}{2}}}(z) \subset B_{r}(z)$ and hence to get (\ref{GT}) we just need to observe that 
$$ \sup _{B_{r/2}(y)} | \psi_{ \lambda}|^2 \leq \sup_{z \in B_{r/2}(y)} \sup _{B_{ \frac{a_0}{2}\lambda^{-\frac{1}{2}}}(z)} | \psi_{ \lambda}|^2  \leq b \lambda^{n/2} \sup_{z \in B_{r/2}(y)} \int_{B_r(z)} | \psi_{ \lambda}|^2, $$with $b= b_0a_0^{-n} $.
Now we apply (\ref{Denominator}) and (\ref{LocalLinfty}) to (\ref{Doubling}) to achieve 
$$\delta \in (0, 10r), x \in B_{\frac{R}{2}}(p): \quad \sup _{B_{2 \delta }(x)} | \psi_{\lambda}|^2 \leq d e^{\kappa_1 \, r \sqrt{\lambda}} (r \sqrt{\lambda})^{n\kappa_2} \sup_{B_{\delta}(x)} | \psi_{\lambda}|^2 . $$ for some uniform constant $d$ which depends on $K_1$ and $K_2$. We note since $r \sqrt{\lambda} \geq 1$,  if we choose $M$ to be an integer larger than $ \kappa_1$ and $n \kappa_2$ then 
$$ (r \sqrt{\lambda})^{n \kappa_2}e^{\kappa_1 \, r \sqrt{\lambda}}\leq M! e^{2M r \sqrt{\lambda}}.$$ Finally by choosing  $$c \geq \max (\log d, M \log M, 2M),$$ we get (\ref{LinftyDoublingEstimate}).

To prove (\ref{L2DoublingEstimate}) we use (\ref{LinftyDoublingEstimate}). It is enough to show that 
$$ \frac{\int _{B_{2 \delta }(x)} | \psi_{\lambda}|^2}{\int _{B_{\delta }(x)} | \psi_{\lambda}|^2}    \leq  K (\delta \sqrt{\lambda})^n \frac{\sup _{B_{2 \delta }(x)} | \psi_{\lambda}|^2}{\sup _{B_{ \delta/2 }(x)} | \psi_{\lambda}|^2},$$ because $(\delta \sqrt{\lambda})^n \leq (10r \sqrt{\lambda})^n \leq e^{c r \sqrt{\lambda}}$ for some appropriate $c$ as we found in the above argument. The above ratios comparison follows from the trivial estimate
$$\int _{B_{2 \delta }(x)} | \psi_{\lambda}|^2 \leq \frac{1}{a} \, (2\delta) ^n \sup _{B_{2 \delta }(x)} | \psi_{\lambda}|^2,$$ applied to the numerator, and the estimate 
$$ \int _{B_{ \delta }(x)} | \psi_{\lambda}|^2 \geq \frac{1}{b_0} \, (\min ( \lambda^{-1/2}, \delta /4 ))^{n} \sup _{B_{ \delta/2 }(x)} | \psi_{\lambda}|^2, $$ applied to the denominator. The last estimate follows from the elliptic estimate (\ref{EllipticEstimate}) by setting $s= \min (a_0 \lambda^{-1/2}, \delta /4 )$ and writing 
$$ \sup _{B_{\delta/2}(x)} | \psi_{ \lambda}|^2 \leq \sup_{z \in B_{\delta/2}(x)} \sup _{B_{s/2}(z)} | \psi_{ \lambda}|^2  \leq b_0 s^{-n} \sup_{z \in B_{\delta /2}(x)} \int_{B_{s}(z)} | \psi_{ \lambda}|^2 \leq b_0 s^{-n} \ \int_{B_{\delta}(x)} | \psi_{ \lambda}|^2.$$ 

The proofs of (\ref{L2LowerBound}) and (\ref{LinftyLowerBound}) are obtained by iterations of inequalities (\ref{L2DoublingEstimate}) and (\ref{LinftyDoublingEstimate}). Since they are very similar we only give the proof of (\ref{LinftyLowerBound}). Fix $ \delta \leq r/2$ and let $m$ be the greatest integer such that $2^{m-1} \delta \leq r$. Then if we write inequalities (\ref{LinftyDoublingEstimate}) for $\delta, 2 \delta, 4 \delta, \dots 2^{m-1} \delta$ and multiply them all we get 
$$ \sup_{B_\delta(x)} |\psi_\lambda|^2 \geq e^{-mcr \sqrt{ \lambda}}
 \sup_{B_{2^{m}\delta}(x)} |\psi_\lambda|^2.$$ Because of the choice of $m$, we have $ 2^{m} \delta>r$. Hence 
$$ \sup_{B_\delta(x)} |\psi_\lambda|^2 \geq e^{-mcr \sqrt{ \lambda}}
 \sup_{B_r(x)} |\psi_\lambda|^2 \geq \frac{e^{-mcr \sqrt{ \lambda}}}{\text{Vol}(B_r(x))}\int_{B_r(x)} |\psi_\lambda|^2 \geq a K_2 e^{-mcr \sqrt{ \lambda}}. $$  Since $m \geq \log (r/ \delta)$ and $r \sqrt{ \lambda} \geq 1$, by selecting $c$ slightly larger  the lower bound (\ref{LinftyLowerBound}) follows.

 \end{proof}

\subsection{Second proof of improved  $L^\infty$-growth estimates (\ref{LinftyDoublingEstimate})} We recall the following result of \cite{Ma}, which is similar to estimate (\ref{Doubling}).  
\begin{theo}[Mangoubi, \cite{Ma} Theorem 3.2]
Let $(X, g)$ be a smooth Riemannian manifold, $p \in X$, and $R>0$ so that the geodesic ball $B_{2R}(p)$ is embedded, and denote $S=\sup_{B_{2R(p)}} | \text{Sec}(g)|$. Suppose $\psi_\lambda$ is a smooth function such that $\Delta_g \psi_\lambda = \lambda \psi_\lambda$ on $B_{2R}(p)$ for some $\lambda \geq 0$. Then for all 
$ \delta  \leq s \leq \min (C S^{-1/2}, R/6)$, and all $x \in B_{R/2}(p)$
$$\sup _{B_{3 \delta }(x)} | \psi_{\lambda}|^2 \leq c_0 e^{c_1 \, s \sqrt{\lambda}} \left ( \frac{\sup _{B_{ 3 s }(x)} | \psi_{ \lambda}|^2 }{  \sup _{B_{s}(x)} | \psi_{\lambda}|^2 }    \right )^{1+c_2 \delta^2 S} \sup_{B_{2\delta}(x)} | \psi_{\lambda}|^2, $$ where $C$, $c_1$ and $c_2$ are positive constants which depend only on $R$, and $c_0$ depends on bounds on $(g^{-1})_{ij}$, its derivatives and its ellipticity constant on the ball $B_{2R}(p)$. 
\end{theo}
Using this theorem twice, we get for $\delta \leq s \leq \min (C S^{-1/2}, R/6)$
\begin{align*} 
\frac {\sup _{B_{2 \delta }(x)} | \psi_{\lambda}|^2}{\sup_{B_{\delta}(x)} | \psi_{\lambda}|^2 }& \leq  
 \frac {\sup _{B_{\frac{9}{4} \delta }(x)} | \psi_{\lambda}|^2}{\sup_{B_{\frac{3}{2}\delta}(x)} | \psi_{\lambda}|^2 } \frac {\sup _{B_{\frac{3}{2} \delta }(x)} | \psi_{\lambda}|^2}{\sup_{B_{\delta}(x)} | \psi_{\lambda}|^2 }\\ 
& \leq c_0^2 e^{2c_1 \, s \sqrt{\lambda}} \left ( \frac{\sup _{B_{ 3 s }(x)} | \psi_{ \lambda}|^2 }{  \sup _{B_{s}(x)} | \psi_{\lambda}|^2 }    \right )^{2+c'_2\delta^2 S}, \end{align*} 

for a new constant $c'_2$. Now we choose $r_0(g) \leq \frac{1}{10}\min (C S^{-1/2}, R/6),$ we  put $s=10r$, and argue as we did following inequality (\ref{Doubling}).  

\subsection{Proof of (\ref{RefinedOV}); upper bound on the order of vanishing} Let us show that the upper bound (\ref{RefinedOV}) on the order of vanishing $ \nu_x(\psi_\lambda)$ follows from the lower bound (\ref{LinftyLowerBound}). Suppose $\psi_\lambda$ vanishes at $x$ to order $M$. Then there exists $\delta_0>0$ such that for all $\delta < \delta_0$
$$ C_{\psi_\lambda, \delta_0} \delta^M \geq \sup_{B_{\delta}(x)} |\psi_\lambda|^2.$$ Therefore using (\ref{LinftyLowerBound}), for all $ 0<\delta < \min(\delta_0, r/2)$
$$ C_{\psi_\lambda, \delta_0} \delta^M \geq \Big (\frac{\delta}{r} \Big )^{c r \sqrt{\lambda}}. $$ Dividing by $\delta^M$ and letting $ \delta \to 0$ we see that we must have $ M \leq c r\sqrt{\lambda}$. 

\subsection{Proof of (\ref{RefinedUB}); upper bounds on the size of nodal sets for $n \geq 3$.}
The main tool is the following result of \cite{La}. 
\begin{theo}[Logunov \cite{La}, Theorem 6.1] Let $(\tilde X, \tilde g)$ be a smooth Riemannian manifold of dimension $d$, $\tilde p \in \tilde X$, and $\tilde R>0$ so that the geodesic ball $B_{2\tilde R}(\tilde p)$ is embedded. Suppose $H$ is a harmonic function on $B_{2\tilde R}(\tilde p)$; that is $\Delta_{\tilde g} H = 0$ on $B_{2\tilde R}(\tilde p)$.  Then there exists $R_0=R_0 (B_{2 \tilde R}(\tilde p), g)< \tilde R$ such that for any Euclidean \footnote{It means that $\tilde Q$ is a cube in the chart associated to the geodesic normal coordinates at $\tilde p$. } cube $\tilde Q \subset B_{ R_0}( \tilde p)$ one has
$$ \mathcal H^{d-1} \Big(\{H=0 \} \cap \tilde Q \Big )  \leq \kappa \, \text{diam}(\tilde Q)^{d-1} N(H,\tilde Q)^{2\alpha}, $$ for some $\alpha > \frac{1}{2}$ that is only dependent on $d$, and some $\kappa$ that depends only on $(B_{2 \tilde \tilde R}(\tilde p), g)$. Here, 
$$ N(H, \tilde Q) = \sup_{B^e_{s}(x) \subset \tilde Q} \log \left ( \frac{\sup_{B^e_{2s}(x)}|H|^2}{\sup _{B^e_{s}(x)}|H|^2 } \right ), $$
where $B^e_s(x)$ stands for the Euclidean ball of radius $s$ centered at $x$ in the normal chart of $B_{\tilde R}( \tilde p)$.   
\end{theo} 

To prove (\ref{RefinedUB}), we use our modified growth estimates (\ref{L2DoublingEstimate}) and the above theorem. We first cover $(X, g)$ using geodesic balls $\{B_{r}(x_i) \}_ {x_i \in \mathcal I}$ such that each point in $X$ is contained in $C(X, g)$ many of the double balls $\{B_{2r}(x_i) \}_ {x_i \in \mathcal I}$, where $C(X,g)$ is independent of $r$ and depends only on the injectivity radius of $(X, g)$ and a bound on the Ricci curvature of $(X, g)$. Such a thing is possible by Bishop-Gromov volume comparison theorem. For a proof see for example \cite{CM} Lemma 2. It is then easy to see that such a covering has at most $C_0 r^{-n}$ open balls for some uniform constant $C_0=C_0(X,g)$. Next we estimate $\mathcal H^{n-1}(Z_{\psi_\lambda} \cap B_{r}(p))$ for each $p\in \mathcal I$. To do this we define
$$ \tilde X =X \times \mathbb R, \;  d=n+1, \; \tilde g=\text{product metric},  \;\tilde p =(p, 0),  \;\tilde Q_r (\tilde p) =  Q_r (p) \times [-r, r], $$ where $Q_r(p)$ is the Euclidean cube of sidelengths $2r$ centered at $p$. We shall also use $ \tilde x = (x, t)$. We then put
$$ H(\tilde x)= \psi_\lambda (x) e^{ t\sqrt{\lambda}}.$$ Then clearly $\Delta_ {\tilde g} H =0$. We also observe that 
$$ B_r(p) \subset Q_r(p)\quad \text{and} \quad  \tilde Q _r (\tilde p) \subset B_{R_0} ( \tilde p).$$ We can in fact choose $R_0$ to be independent of $ \tilde p$  and $ \tilde R$ because $R_0$  in the above theorem is non-increasing in the sense that $R_0 (B_{2 \tilde R_1}(\tilde p_1), g) \geq R_0 (B_{2 \tilde R_2}(\tilde p_2), g)$ whenever $B_{2 \tilde R_1}(\tilde p_1) \subset B_{2 \tilde R_2}(\tilde p_2)$. Then a uniform $R_0$ can be chosen by means of the \textit{Lebesgue number lemma} and the compactness of $X \times [-1, 1]$. We also need to make sure that $r < \frac{R_0}{2}$. This can be made possible by choosing $r_0(g)$ in Theorem \ref{UpperBound} smaller if necessary. We then write 
\begin{align*} \mathcal H^{n-1}(\{ \psi_\lambda =0 \} \cap B_{r}(p)) & \leq \mathcal H^{n-1}(\{ \psi_\lambda =0 \} \cap Q_r(p)) \\ &= \frac{1}{2r} \mathcal H^{n}(\{ H =0 \} \cap \tilde Q_r (\tilde p)) \\
& \leq \frac{\kappa}{2r} (2r)^n N(H, \tilde Q _r(\tilde p))^{2\alpha} \\
& = \kappa ' r^{n-1} N(H, \tilde Q _r(\tilde p))^{2\alpha}\end{align*}
Now we use our doubling estimates to show that $N(H, \tilde Q _r(\tilde p)) \leq c' r \sqrt{\lambda}$ for some $c '$ that is uniform in $r$, $\lambda$, and $p$. We underline that our doubling estimates involve geodesic balls, but the definition of the doubling index $N$ in \cite{La} uses Euclidean balls $B^e_s( \tilde x)$ in a fixed normal chart of $B_{2 \tilde R}( \tilde p)$. However, by choosing $R_0$ sufficiently small we can make sure that 
$$ B_{s/2}( \tilde x) \subset B^e_s( \tilde x) \subset B_{3s/2}( \tilde x)$$ for all $\tilde x \in B_{R_0}( \tilde p)$ and all $s < R_0$. As a result of this if we assume $ r <\frac{R_0}{10}$, then using (\ref{LinftyDoublingEstimate}) four times we get
\begin{align*} N(H, \tilde Q_r(\tilde p) ) & = \sup_{B^e_{s}( \tilde x) \subset \tilde Q_r (\tilde p)} \log \left ( \frac{\sup_{B^e_{2s}( \tilde x)}|H(\tilde x)|^2}{\sup _{B^e_{s}(\tilde x)}|H(\tilde x)|^2 } \right )\\ & \leq \sup_{B_{s/2}(\tilde x) \subset \tilde Q_r (\tilde p)} \log \left ( \frac{\sup_{B_{3s}(\tilde x)}|H(\tilde x)|^2}{\sup _{B_{s/2}(\tilde x)}|H(\tilde x)|^2 } \right ) \\
& \leq \sup_{B_{s/2}(x) \subset Q_{r}(p)} \log \left (e^{5s \sqrt{\lambda}} \frac{\sup_{B_{3s}(x)}|\psi_\lambda(x)|^2}{\sup _{B_{s/4}(x)}|\psi_\lambda(x)|^2 } \right ) \\
& \leq c' r \sqrt{\lambda}. \end{align*}
Finally
\begin{align*} \mathcal H^{n-1} ( Z_{\psi_\lambda} )  \leq \sum_{ x_i \in \mathcal I} \mathcal  H^{n-1}(Z_{\psi_\lambda} \cap B_{r}(x_i)) & \leq  C_0 r^{-n} \kappa' r^{n-1} (c'^2r^2 \lambda)^\alpha  \leq c_1 r^{2\alpha -1} \lambda^ \alpha , \end{align*} for some $c_1$ that is uniform in $r$ and $\lambda$.

\subsection{Proof of (\ref{RefinedUB2}); Upper bounds on the size of nodal sets for surfaces}
The main tool is the following local result of \cite{LM}.

\begin{theo}[Logunov-Malinnikova \cite{LM}] \label{LM} 
Let $(\tilde X, \tilde g)$ be a smooth Riemannian manifold of dimension $n=2$, $p \in \tilde X$ a point, and $\tilde R >0$ a radius such that the $\tilde g$-geodesic ball $\tilde B_{2 \tilde R}(p)$ is embedded. Let $\psi_{\tilde \lambda}$ be a smooth function such that for some $\tilde \lambda \geq 1$ we have $\Delta_{\tilde g} \psi_{\tilde \lambda} = \tilde \lambda \psi_{\tilde \lambda}$ on $\tilde B_{2 \tilde R}(p)$.  Suppose we also know that there exists some $s_0 \leq \frac{R}{10}$ such that for all $s < s_0$ we have
$$ \frac{\sup_{\tilde B_{2 s} (x)} | \psi_{\tilde \lambda} |^2}{\sup_{ \tilde B_{ s} (x)} | \psi_{\tilde \lambda} |^2} \leq C_1 e^{ c \sqrt{\tilde \lambda}},$$ for some constants $c$ and $C_1$ that are uniform for $x \in \tilde B_{\tilde R} (p)$. Then
\begin{equation} \mathcal H_{\tilde g} ^1 \Big ( \{ \psi_{\tilde \lambda} =0 \} \cap \tilde B_{\tilde R/2}(p) \Big ) \leq C_2 {\tilde \lambda} ^{\frac{3}{4} -\beta },\end{equation}
where $\beta \in (0, \frac{1}{4})$ is a small universal constant and $C_2$ is controlled by $c$, $C_1$, and the $\mathcal C^k$ norm of $ (\tilde g^{-1})_{ij}$ on $\tilde B_{2 \tilde R} (p)$ for some universal $k$.  
\end{theo}

To prove (\ref{RefinedUB2}), suppose $\psi_\lambda$ is an eigenfunction of $\Delta_g$ on $(X, g)$. We cover $X$ by geodesic balls $\{B_{r/2}(x_i) \}_{x_i \in \mathcal I}$ of radius $\frac{r}{2}$ in such a way that the number of them is at most $C_0 r^{-n}$. As we saw earlier this is always possible. We then estimate the size of the nodal set of $\psi_\lambda$ in each $B_{r/2} (x)$ using Theorem \ref{LM}.  To do this, we first define $(\tilde X, \tilde g)= (X, \frac{1}{r^2}g)$. Under such a rescaling, a ball of radius $r$ scales to a ball of radius $1$. Hence we put $\tilde R =1$. Then the equation  $$-\Delta_g \psi_\lambda = \lambda \psi_\lambda \qquad  \text{on} \quad B_{2r}(p),$$ becomes 
$$-\Delta_{\tilde g} \psi_{\lambda} = \tilde \lambda \psi_{ \lambda} \qquad \text{on} \quad \tilde B_{2}(p),$$ with $$\tilde \lambda = r^2\lambda \qquad \text{and} \qquad \psi_{\tilde \lambda} =\psi_\lambda \; . $$ We can see that the doubling condition of Theorem \ref{LM} is valid because for all $s \leq \frac{1}{10}$, using (\ref{LinftyDoublingEstimate})
$$ \frac{\sup_{\tilde B_{2 s} (x)} | \psi_{ \lambda} |^2}{\sup_{ \tilde B_{s} (x)} | \psi_{ \lambda} |^2} = \frac{\sup_{B_{2 s r} (x)} | \psi_{ \lambda} |^2}{\sup_{ B_{s r} (x)} | \psi_{\lambda} |^2}  \leq e^{c r \sqrt{\lambda}}= e^{ c \sqrt{\tilde \lambda}}, $$ for some $c$ that is uniform in $ \tilde \lambda$, $s$, and $x$, and is controlled by $K_1$, $K_2$, and the $\mathcal C^k$ norm of $ ( \tilde g)^{ij}$ on $ \tilde B_{2} (p)$ for some universal $k$.  Therefore, by Theorem \ref{LM} 
$$ \mathcal H_{g} ^1 \Big ( \{ \psi_{\lambda} =0 \} \cap B_{r/2}(p) \Big ) =r^{n-1} \mathcal H_{\tilde g} ^1 \Big ( \{ \psi_{\tilde \lambda} =0 \} \cap \tilde B_{1/2}(p) \Big ) \leq C_2 r^{n-1}  {\tilde \lambda} ^{\frac{3}{4} -\beta }.$$ We emphasize that since $( \tilde g)^{ij} =r^2 g^{ij}$,  for small enough $r_0(g)$ and all $r < r_0(h)$, the $\mathcal C^k$ norm of $ (\tilde g)^{ij}$ on $ \tilde B_{2} (p)$ is bounded by the $\mathcal C^k$ norm of $(g)^{ij}$ on $B_{2r}(p)$. Hence $ C_2$ is independent of $r$, $\lambda$, and $p$,  and  is controlled only by $K_1$ and $K_2$ and $(X, g)$. 
Adding these up,  we get
\begin{align*} \mathcal H_{g} ^1 \Big ( \{ \psi_{\lambda} =0 \} \Big ) \leq  \sum_{x_i \in \mathcal I} \mathcal H_{g} ^1 \Big ( \{ \psi_{\lambda} =0 \} \cap B_{r/2}(x_i) \Big ) & \leq (C_0 r^{-n}) C_2  r^{n-1}  {\tilde \lambda} ^{\frac{3}{4} -\beta } \\
& =c_3 r^{1-2 \beta} \lambda^{\frac{3}{4} - \beta}. \end{align*}

\subsection{Proof of (\ref{RefinedOV2}); Number of singular points for surfaces}
We shall use the results of Dong \cite{Dong} instead of \cite{DF2}, although both methods would work.  Another goal is simplify a less detailed part of the argument of \cite{Dong}. Let us first recall some statements from \cite{Dong}. 

\begin{theo} [Dong \cite{Dong}, Theorems 2.2 and 3.4] Let $(X, g)$ be a smooth Riemannian manifold of dimension $2$, $p \in X$, and $R>0$ so that the geodesic ball $B_{2R}(p)$ is embedded. Suppose $\psi_\lambda$ is a smooth function such that $\Delta_g \psi_\lambda = \lambda \psi_\lambda$ on $B_{2R}(p)$ for some $\lambda \geq 1$. Denote
$$ q_\lambda= | \nabla \psi_\lambda|^2 + \frac{\lambda}{2} |\psi_\lambda | ^2. $$
Then for all $x \in B_{R/2}(p)$ and all $s < \frac{R}{8}$
\begin{equation}\label{DongOV}
\quad  \sum_{z \in Z_{\psi_\lambda} \cap B_{s}(x)}  \left ( \nu_z(\psi_\lambda)-1 \right) \leq \alpha_1 \sqrt{\lambda} + \alpha_2 s^2 \lambda.
 \end{equation}
The constants $\alpha_1, \alpha_2$ are inform in $x$, $s$, and $\lambda$, and depend only on  $(B_{2R}(p), g)$. 
\end{theo}
In fact by a glance at 
the proof of \ref{DongOV} (see Theorem 3.4 of \cite{Dong}, pages 502-503), one sees that the following statement holds:
\begin{equation}\label{DongOVRefined}
\quad  \sum_{z \in Z_{\psi_\lambda} \cap B_{s}(x)}  \left ( \nu_z(\psi_\lambda)-1 \right) \leq \alpha'_1 \log \left ( \frac{\sup_{B_{4s}(x)} q_\lambda}{\sup_{B_s(x)} q_\lambda} \right )+ \alpha_2 s^2 \lambda,
\end{equation}
for some uniform constants $ \alpha'_1$ and $\alpha_2$. 

The estimate (\ref{DongOV}) follows quickly from  (\ref{DongOVRefined}) if one knows that 
$$ s \in \Big (0, \frac{R}{8}  \Big ), x \in B_{\frac{R}{2}}(p): \quad  \frac{\sup_{B_{4s}(x)} q_\lambda}{\sup_{B_s(x)} q_\lambda} \leq \alpha_3 e^{c_2 \sqrt{\lambda}}.$$
The above growth estimate is proved in \cite{Dong} using the theory of frequency functions and monotonicity formulas (see \cite{GaLi, HaL, Lin} for background). However the proof of the monotonicity formula associated to $q_\lambda$ (see pages 498-499) is carried out only for the Euclidean metric and the proof of the upper bound $\sqrt{\lambda}$ on the frequency function is referred to the methods of \cite{Lin}.  Here we give a simpler proof of this growth estimate which is based on gradient estimates for solutions of elliptic equations. More precisely, we show that if doubling estimates (\ref{LinftyDoublingEstimate})
$$ s \in (0, 10r), x \in B_{\frac{R}{2}}(p): \quad \sup _{B_{2 s}(x)} | \psi_{\lambda}|^2 \leq e^{c \, r \sqrt{\lambda}} \sup_{B_{s}(x)} | \psi_{\lambda}|^2 $$
hold, then
\begin{equation} \label{GrowthEstimateQ} s \in ( \lambda^{-\frac{1}{2}}, 2r ), x \in B_{\frac{R}{2}}(p): \quad \frac{\sup_{B_{4s}(x)} q_\lambda}{\sup_{B_s(x)} q_\lambda} \leq \alpha_3 e^{c_2 r \sqrt{\lambda}}, \end{equation} for uniform constants $\alpha_3$ and $c_2$. For the proof we use an application of standard elliptic estimates to the gradient of eigenfunctions, as performed in \cite{SX}.
\begin{theo} [\cite{SX}, Theorem 1] Let $(X, g)$ be a smooth connected compact Riemannian manifold without boundary. Suppose $\psi_ \lambda$ is an eigenfunction of $\Delta_g$ with eigenvalue $\lambda$. Then
$$ \beta_1 \sqrt{\lambda} \sup_X | \psi_\lambda| \leq \sup_X | \nabla \psi_\lambda | \leq \beta_2 \sqrt{\lambda} \sup_X | \psi_\lambda|, $$ for some positive constants $\beta_1$ and $\beta_2$ independent of $\lambda$. 
\end{theo} 
In fact by looking at the proof of this theorem we notice that a stronger statement holds. More precisely, one can see that (see page 23, Fact (1) and Eq (6)) for all $s < \text{inj}(g)/4$
$$ \beta_1  \sqrt{\lambda} \sup_{B_{s}(x)} |\psi_\lambda | \leq \sup_{B_{s+  \frac{\gamma_0}{ \lambda^{{1/2}}}}(x)} | \nabla \psi_\lambda | $$
\begin{equation}\label{SXUB} \sup_{B_{s}(x)} | \nabla \psi_\lambda | \leq \beta_2 \sqrt{\lambda} \sup_{B_{s+\frac{1}{\lambda^{ {1/2}}}}(x)} | \psi_\lambda |, \end{equation}
where $\gamma_0$ is a positive constant that depends only on the Riemannian manifold $(X ,g)$. In fact it is the Br\"uning constant that guarantees that in every ball of radius $\frac{\gamma_0}{ \lambda^{{1/2}}}$ there is a zero of $ \psi_\lambda$.  However, to prove (\ref{GrowthEstimateQ}) we only need the upper bound (\ref{SXUB}) for the gradient \footnote{This is proved easily by a rescaling argument and elliptic estimates such as Theorem 8.32 in Gilbarg-Trudinger \cite{GT}}. Let $s \in ( \lambda^{-\frac{1}{2}}, 2r)$. Then since $4s +\lambda^{-1/2} < 10r$, using our doubling estimates (\ref{LinftyDoublingEstimate}) three times, we get
\begin{align*} \sup_{B_{4s}(x)} q_\lambda & = \sup_{B_{4s}(x)} \Big ( | \nabla \psi_\lambda|^2 + \frac{\lambda}{2} |\psi_\lambda | ^2 \Big ) \\
 & \leq  \beta'_2 \lambda \sup_{B_{4s+ \frac{1}{\lambda^{1/2}}}(x)} | \psi_\lambda|^2 \\
& \leq \beta'_2 \lambda  e^{3 c r \sqrt{\lambda}} \sup_{B_{\frac{s}{2}+ \frac{1}{8\lambda^{{1/2}}}}(x)} | \psi_\lambda|^2\\
& \leq 2 \beta'_2  e^{3 c r \sqrt{\lambda}}  \sup_{B_s(x)} q_\lambda. 
\end{align*} This proves (\ref{GrowthEstimateQ}) with $\alpha_3= 2 \beta'_2$ and $c_2 =3c$. 

To finish the proof of our upper bounds for the number of singular points for surfaces, we apply (\ref{GrowthEstimateQ}) to the inequality (\ref{DongOVRefined}) and obtain
$$ \sum_{z \in Z_{\psi_\lambda} \cap B_{s}(x)}  \left ( \nu_z(\psi_\lambda)-1 \right) \leq \alpha''_3 r \sqrt{\lambda} + \alpha_2 s^2 \lambda. $$ We now put $s= r^{\frac{1}{2}} \lambda^{-\frac{1}{4}}$. We underline that this choice of $s$ is in fact in the allowable range $(\lambda^{-\frac{1}{2}}, 2r)$ because $ r \geq \lambda^{-\frac{1}{2}}$. From this, (\ref{RefinedOV2}) follows immediately. 


\subsection{Proof of Theorem (\ref{UpperBoundQE}); Upper bounds for QE eigenfunctions}
This theorem follows quickly from the lemma below combined with Theorem \ref{UpperBound}. 
\begin{lem} \label{QElemma} Let $\{ \psi_j \}_{j \in S}$ be a  sequence of eigenfunctions of $\Delta_g$ with eigenvalues $\{\lambda_j\}_{j \in S}$ such that for all $r \in (0, \text{inj}(g)/2)$ and all $x \in X$
\begin{equation} \label{QEonX2}
\int_{B_r(x)} |\psi_{j}|^2 \to \frac{\text{Vol}_g(B_r(x))}{\text{Vol}_g(X)}, \qquad  \lambda_j \xrightarrow{j \in S} \infty.
\end{equation}
Then there exists $r_0(g)$ such that for each $r \in (0, r_0(g))$ there exists $\Lambda_r$ such that for $ \lambda_j \geq \Lambda_r$ we have
$$K_1 r^n\leq \int _{B_{r}(x)} | \psi_{j}|^2 \leq K_2 r^n, $$ uniformly for all $x \in X$. Here, $K_1$ and $K_2$ are independent of $r$, $j$, and $x$. 
 \end{lem}
We point out that this lemma is obvious when $x$ is fixed, however to obtain uniform $L^2$ estimates we need to use a covering argument as follows. 

\begin{proof} First we choose $r_0(g) < \frac{\text{inj}(g)}{4}$ small enough so that for all $ r <r_0(g)$
$$ a_1 r^n \leq \text{Vol}(B_{r/2}(x)) <  \text{Vol}(B_{2r}(x))  \leq a_2 r^n,$$ for some positive $a_1$ and $a_2$ that are independent of $r$ and $x$. Next, we cover $(X, g)$ using geodesic balls $\{B_{r/2}(x_i) \}_ {x_i \in \mathcal I}$ such that $\text{card}\, (\mathcal I)$ is at most $C_0r^{-n}$, where $C_0$ depends only on $(X, g)$. The existence of such a covering was discussed in the proof of (\ref{RefinedUB}).  For each $x_i \in \mathcal I$, by using (\ref{QEonX2}) twice, we can find $\Lambda_{i, r}$ large enough so that for $\lambda_j \geq \Lambda_{i, r}$
$$K_1 r^n\leq \int _{B_{r/2}(x_i)} | \psi_{j}|^2  \leq \int _{B_{2r}(x_i)} | \psi_{j}|^2 \leq K_2 r^n, $$
with $K_1= \frac{a_1}{2\text{Vol}(X)}$  and $K_2= \frac{2a_2}{\text{Vol}(X)}$. We claim that $ \Lambda_r =\max_{i \in \mathcal I} \{ \Lambda_{i, r} \}$ would do the job for all $x$ in $X$. So let $x$ be in $X$ and $r$  be as above. Then $x \in B_{r/2}(x_i) $ for some $i \in \mathcal I$ and clearly one has $B_{r/2}(x_i) \subset  B_r(x) \subset B_{2r}(x_i) $.  This and the above inequalities prove the lemma. 
\end{proof}

\section*{Acknowledgments}

We are grateful to Gabriel Rivi\`ere and Steve Zelditch for their helpful comments on the earlier version of this paper.

\end{document}